\documentclass{amsart}
\usepackage{graphicx}
\usepackage{color}
\usepackage{enumerate,url,amssymb,mathrsfs,upref,pdfsync}

\theoremstyle{plain}
\newtheorem{theorem}{Theorem}
\newtheorem{lemma}[theorem]{Lemma}

\theoremstyle{definition}

\theoremstyle{remark}
\newtheorem{remark}[theorem]{Remark}

\makeatletter
\renewcommand*\subjclass[2][2000]{%
  \def\@subjclass{#2}%
  \@ifundefined{subjclassname@#1}{%
    \ClassWarning{\@classname}{Unknown edition (#1) of Mathematics
      Subject Classification; using '1991'.}%
  }{%
    \@xp\let\@xp\subjclassname\csname subjclassname@#1\endcsname
  }%
}

\def\d#1{{#1\kern-0.4em\char"16\kern-0.1em}}
\def\D#1{{\raise0.2ex\hbox{-}\kern-0.4em #1}}
\def\dj{d\kern-0.4em\char"16\kern-0.1em}
\def\Dj{\mbox{\raise0.3ex\hbox{-}\kern-0.4em D}}
\newcounter{zd}

\newcounter{zdr}[subsection]

\def\cal{\mathcal}

\begin{document}
\title[]{NORM OF THE BERGMAN PROJECTION ONTO THE BLOCH SPACE}

\author{David Kalaj}
\address{Faculty of natural sciences and mathematics, University of Montenegro,
 D\v zor\v za Va\v singtona b.b. 81000, Podgorica, Montenegro}
\email{davidkalaj@gmail.com}

 \author{\Dj OR\Dj IJE VUJADINOVI\'C }
\address{Faculty of natural sciences and mathematics, University of Montenegro, D\v zor\v za Va\v singtona b.b. 81000 Podgorica, Montenegro}
 \email{  djordjijevuj@t-com.me}

\begin{abstract}
We consider weighted Bergman projection $P_{\alpha}: L^{\infty}(\Bbb B) \rightarrow {\cal B} $ where $\alpha>-1$ and $\cal B$ is the Bloch  space of the unit ball $\Bbb B$ of the complex space $\Bbb C^n.$   We obtain the exact norm of the operator $P_{\alpha}$ where the Bloch space is observed as a space  with norm (and semi-norm) induced  from  the Besov space $B_{p},0<p<\infty,(B_{\infty}=\cal B).$  Our work contains, as a special case, the   main results from \cite{Kalaj} and \cite{Ant}.
 \end{abstract}



\maketitle

\section{Introduction and Preliminaries}

Throughout this paper we denote by $\Bbb C^n$  complex $n-$dimensional space. Here $n$ is an integer greater than or equal to 1. As  usually $\left<\cdot,\cdot\right>$ represents the inner product in $\Bbb C^n,$
$$\left<z,w\right>=z_1\bar{w}_1+\cdot\cdot\cdot+z_{n}\bar{w}_n,\enspace z,w\in\Bbb C^n, $$
where $z=(z_1,...,z_n)$ and $w=(w_1,...,w_n)$ are coordinate representations in the standard base $\{e_1,...,e_n\}$ of $\Bbb C^n.$
The Euclidean norm in $\Bbb C^n$ is given by
$$|z|=\left<z,z\right>^{\frac{1}{2}}.$$
Let us denote by $\Bbb B$  the unit ball in $\Bbb C^n,\enspace\Bbb B= \{z:|z|<1\}$ and $\Bbb S$ denotes its boundary.

The volume measure $dv$ in $\Bbb C^n$ is normalized, i.e. $v(\Bbb B)=1.$ Also, we are going to treat a class of weighted measures  $dv_{\alpha}$ on $\Bbb B,$ which are defined by
$$dv_{\alpha}(z)=c_{\alpha}(1-|z|^2)^{\alpha}dv(z),\enspace z\in \Bbb B $$
where $\alpha>-1,$ and $c_{\alpha}$ is a constant such that $v_{\alpha}(\Bbb B)=1.$ Direct calculation gives:
$$c_{\alpha}=\frac{\Gamma(n+\alpha+1)}{n!\Gamma(\alpha+1)}.$$

We let $\sigma$ be unitary-invariant positive Borel measure on $\Bbb S$ for which
$\sigma(\Bbb S)=1.$ The term "unitary-invariant" refers to the unitary transformations of $\Bbb C^{n}.$ More precisely, if $U$ is unitary transformation of $\Bbb C^{n},$ then for any $f\in L^{1}(\Bbb S,d\sigma),$
$$\int_{\Bbb S}f(U\zeta)d\sigma(\zeta)=\int_{\Bbb S}f(\zeta)d\sigma(\zeta).$$
The automorphism group of $\Bbb B$ denoted by ${\bf Aut}(\Bbb B)$ consists of all bi-holomorphic mappings of $\Bbb B$ (see~\cite{Rudin1}). A special class of automorphism group are involutive automorphisms, which are, for any point $a\in \Bbb B,$ defined as
$$\varphi_{a}(w)=\frac{a-\frac{\left<w,a\right>a}{|a|^2}-\sqrt{1-|a|^2}(w-\frac{\left<w,a\right>a}{|a|^2})}{1-\left<z,a\right>},\enspace w\in \Bbb B. $$
When $a=0,$ we define $\varphi_{a}=-{\bf Id_{\Bbb B}}.$
We should observe that, $\varphi_{a}(0)=a$ and $\varphi_{a}\circ\varphi_{a}={\bf  Id_{\Bbb B}}.$

In the case when we treat $\Bbb C^n$ as the real $2n-$dimensional space $\Bbb R^{2n},$ the real Jacobian of $\varphi_{a}$ is given by
$$(J_{R}\varphi_a)(w)=\left(\frac{1-|a|^2}{|1-\left<a,w\right>|^{2}}\right)^{n+1}.$$
 We are going to use the following  identities ($a\in \Bbb B$)
\begin{equation}
1-|\varphi_{a}(w)|^{2}=\frac{(1-|a|^2)(1-|z|^2)}{|1-\left<w,a\right>|^2},\enspace z\in \Bbb B
\end{equation}
and
\begin{equation}
1-\left<\varphi_{a}(z),\varphi_{a}(w)\right>=\frac{(1-\left<a,a\right>)(1-\left<z,w\right>)}{(1-\left<z,a\right>)(1-\left<a,w\right>)},z,w\in \Bbb B.
\end{equation}
 Traditionally, $H(\Bbb B)$ denotes the space of all holomorphic functions on $\Bbb B$ and the space of all bounded holomorphic functions  is denoted by $H^{\infty}(\Bbb B).$

  The complex gradient of holomorphic function $f\in H(\Bbb B)$ is defined as
  $$\nabla f(z)=\left(\frac{\partial f}{\partial z_1}(z),...,\frac{\partial f}{\partial z_n}(z)\right), \enspace z\in \Bbb B.$$
 The Bloch space ${\cal B}$ contains all holomorphic functions in $\Bbb B$ with finite  semi-norm
 $$\| f\|_{\beta}=\sup_{z\in \Bbb B}{(1-|z|^2)|\nabla f(z)|}.$$
We can obtain proper norm by adding $|f(0)|,$ i.e.
$$\| f\|_{{\cal B}}=|f(0)|+\| f\|_{\beta}.$$
The Bloch space is the Banach space in above norm. More information about the Bloch space reader can find in \cite{Zhu}.

The Bergman projection  operator $P_{\alpha}\enspace (\alpha>-1)$ is a central mapping in the study of analytic function spaces and it is defined as follows:

$$P_{\alpha}f(z)=\int_{\Bbb B}{\cal K_{\alpha}(z,w)}f(w)dv_{\alpha}(w),\enspace f\in L^{p}(\Bbb B,dv_{\alpha}),$$
where $L^{p}(\Bbb B,dv_{\alpha})$ is the Lebesgue space of all measurable functions on $\Bbb B$ in which modulus with exponent
$p\enspace (1\leq p<\infty)$ is integrable on $\Bbb B$  with  respect to the measure $dv_{\alpha}.$  The case $p=\infty$ corresponds to the space of essentially bounded functions in the unit ball. Here
$${\cal K_{\alpha}}(z,w)=\frac{1}{(1-\left<z,w\right>)^{n+1+\alpha}},\enspace z,w\in\Bbb B$$
is the weighted Bergman kernel. Concerning the Bergman projection, the following two questions are of the main interest for research: establishing the boundedness and determining the exact norm.

Here we want to point out that the Bergman projection $P_{\alpha}:L^{\infty}(\Bbb B)\rightarrow {\cal B}$ is bounded and onto (see~\cite{Zhu}).

In the case when $n=1$
 for the semi-norm $\| f\|=\sup_{|z|<1}(1-|z|^2)|f'(z)|,$ Per\"al\"a (see~\cite{Ant}) determined the norm of the Bergman projection. He obtained that  $\|P\|=\sup_{\|f\|\leq 1}\|Pf\|=\frac{8}{\pi}.$   A generalization of this result in the unit ball $\Bbb B\subset \Bbb C^{n}$ was done by Kalaj and Markovi\'c (see~\cite{Kalaj}), where it is shown that  $\|P\|=\frac{\Gamma(n+\alpha+2)}{\Gamma^2(\frac{n+\alpha+2}{2})}$. Later, in work of Per\"al\"a (see \cite{Ant1}), the author completed his earlier result from \cite{Ant} and its generalization in \cite{Kalaj} by finding the norm of the Bergman projection w.r.t. to the proper norm of the Bloch space. We remark that calculating the exact norm of Bergman projection $P$ on $L^
p$ spaces $1<p<\infty$ is a long-standing problem and only partial results are known, see  \cite{Zhu1, Dost}.\\

There are several ways to define norm on the Bloch space, which makes it the Banach space. To this and, let us recall a definition of  the Besov space $B_{p}\enspace(0<p<\infty)$ in unit ball $\Bbb B\subset \Bbb C^n$ (for a reference see~\cite{Zhu}).\\
 The Besov space $B_{p}$ contains all holomorphic functions $f$ in $\Bbb B$ such that norm

\begin{equation}
\| f\|_{B_p}^{p}=\sum_{|m|\leq N-1}\left|\frac{\partial^{|m|}f}{\partial z^{m}}(0)\right|^{p}+\sum_{|m|=N}\int_{\Bbb B}\left|(1-|z|^2)^{N}\frac{\partial^{N}f}{\partial z^{m}}(z)\right|^{p}d\tau(z),
\end{equation}

is finite, where $N$ is a positive integer such that $pN>n.$ The measure $d\tau$ is given by  $$d\tau(z)=\frac{dv(z)}{(1-|z|^2)^{n+1}},\enspace z\in \Bbb B$$
 and multi-index $m,$ represents  $n$-tuples of non-negative integers, $m=(m_1,...,m_n),$ where $|m|=\sum_{i=1}^{n}m_{i}.$\\
  By the semi-norm  $\|\cdot\|_{\beta_p}$ in the Besov space $B_{p}\enspace (0<p<\infty)$  we imply

$$\|f\|_{\beta_{p}}^{p}=\sum_{|m|=N}\int_{\Bbb B}\left|(1-|z|^2)^{N}\frac{\partial^{N}f}{\partial z^{m}}(z)\right|^{p}d\tau(z).$$
When $p=\infty$ the Besov space $B_p$ is the Bloch space, $B_{\infty}={\cal B}.$   We want to define norm (semi-norm) on the Bloch space $B_{\infty}$ induced from the Besov space $B_{p}$ as $ p\rightarrow\infty.$\\

Before we find explicit formula for the norm in the mentioned case,  let us  state short  version of  \cite[Theorem~3.5]{Zhu}.
\begin{theorem}
Suppose $N$ is a positive integer, and $f$ is holomorphic in $\Bbb B,$ then following conditions are equivalent:\\
\vspace{3mm}
(1)$f\in{\cal B}$\\
\vspace{2mm}
(2)The functions $$(1-|z|^2)^{N}\frac{\partial^{N}f}{\partial z^{m}}(z),\enspace |m|=N,$$
are bounded in $\Bbb B.$
\end{theorem}

Now we prove the following lemma:

\begin{lemma}
Let $B_{p},1<p<\infty,$ be the Besov space and $\|\cdot\|_{p}$ is the Besov norm defined by (3) . Then
$$\|f\|_{\beta_{p}}\rightarrow\|f\|_{\tilde{\cal B}},\enspace p\rightarrow\infty\enspace\mbox{provided}\enspace f\in B_{r}\cap{\cal B}\enspace\mbox{for some}\enspace r\in(1,\infty),$$
where
$$\|f\|_{\tilde{\cal B}}=\max_{|m|=N}{\sup_{z\in {\Bbb B}}(1-|z|^2)^{N}\left|\frac{\partial^{N}f}{\partial z^{m}}(z)\right|}.\hspace{10mm} {\rm(1)}$$

\end{lemma}
\begin{proof}
We will prove the lemma in more general setting. Namely, if $\{f_{k}\}_{k=1}^{N}$ is a sequence of measurable functions on the measure space $(\Omega,\mu)$ such that $$f_{k}\in L^{r}(\Omega,\mu)\cap L^{\infty}(\Omega,\mu),\enspace k=1,...,N,\enspace\mbox{for some}\enspace r\in (1,\infty),$$
then
$$\left(\sum_{k=1}^{N}\|f_{k}\|_{p}^{p}\right)^{\frac{1}{p}}\rightarrow\max_{1\leq k\leq N}\|f_k\|_{\infty}.$$
The last relation is an easy consequence of the following relation $\lim_{p\to\infty}\|f_{k}\|_p \to \|f_k\|_\infty$  (see e.g. \cite[p.~73,~Ex.~4]{Rudin}) and
 the following obvious inequalities
 $$\left(\sum_{k=1}^N\|f_k\|^p_p\right)^{1/p}\le N^{1/p}  \max_k\{\|f_k\|_p\},$$
 and
 $$\left(\sum_{k=1}^N\|f_k\|^p_p\right)^{1/p}\geq \max_k\{\|f_k\|_p\}.$$
 It follows that if $f\in B_{r}\cap{\cal B}$ for some $r>1$, then  $$\|f\|_{\tilde{\cal B}}=\lim_{p\to\infty}\|f\|_{\beta_p}= \max_{|m|=N} \sup_{z\in\Bbb B}\left|(1-|z|^2)^{N}\frac{\partial^{N}f}{\partial z^{m}}(z)\right|.$$
\end{proof}
Let us notice that in the same way
$$\left(\sum_{|m|\leq N-1}\left|\frac{\partial^{|m|}f}{\partial z^{m}}(0)\right|^{p}\right)^{\frac{1}{p}}\rightarrow\max_{|m|\leq N-1}{\left|\frac{\partial^{|m|}f}{\partial z^{m}}(0)\right|},\enspace \mbox{as}\enspace p\rightarrow\infty.$$
Thus, we define the proper norm $\|\cdot\|_{\cal B}$ on the Bloch space as follows
\begin{equation}\label{bloc}\|f\|_{\cal B}=\max_{|m|\le N-1}\left|\frac{\partial^{|m|}f}{\partial z^{m}}(0)\right|+ \max_{|m|=N} \sup_{z\in \Bbb B}\left|(1-|z|^2)^{N}\frac{\partial^{N}f}{\partial z^{m}}(z)\right|,\enspace f\in {\cal B},\enspace N\in {\rm{\bf N}},
\end{equation}
and semi-norm $\|\cdot\|_{\tilde{\cal B}}$  is defined as
\begin{equation}\|f\|_{\tilde{\cal B}}= \max_{|m|=N} \sup_{z\in \Bbb B}\left|(1-|z|^2)^{N}\frac{\partial^{N}f}{\partial z^{m}}(z)\right|,\enspace f,\in {\cal B}\enspace N\in {\rm{\bf N}}.
\end{equation}
Although in definition  (3) of the norm $\|\cdot\|_{p}$ for the Besov space $B_{p}$ we have condition $pN>n,$ by the formula (4) we can define $\|\cdot\|_{\cal B}$ on $\cal B$ for any $N.$ This is not surprising because $\infty\cdot N>n$.

The proof of the next lemma is straightforward and we omit the proof.
\begin{lemma}
The Bloch space $\cal B$ is a Banach space in the norm (4)
\end{lemma}

In the sequel, $\tilde{\cal B}$-norm  and ${\cal B}$-norm of the Bergman projection $P_{\alpha}:L^{\infty}\rightarrow {\cal B}$ are
\begin{equation}
\|P_{\alpha}\|_{\tilde{\cal B}}=\sup_{\|g\|_{\infty}\leq 1}\|P_{\alpha}g\|_{\tilde{\cal B}},
\end{equation}
and
\begin{equation}
\enspace\|P_{\alpha}\|_{\cal B}=\sup_{\|g\|_{\infty}\leq 1}\|P_{\alpha}g\|_{\cal B},
\end{equation}
respectively.\\

Now we state the main result of this paper.

\begin{theorem}
 Let $P_{\alpha}$ be the Bergman projection $P_{\alpha}:L^{\infty}(\Bbb B)\rightarrow {\cal B},$ where ${\cal B}$ is the Bloch space with the norm  (4).
Then
 $$\| P_{\alpha}\|_{\tilde{\cal B}}= \frac{\Gamma(n+N+\alpha+1)\Gamma(N)}{\Gamma^{2}(\frac{N}{2}+\frac{n+\alpha+1}{2})}.$$

\end{theorem}

\begin{theorem}
 Let $P_{\alpha}$ be the Bergman projection $P_{\alpha}:L^{\infty}(\Bbb B)\rightarrow {\cal B},$ where ${\cal B}$ is Bloch space in norm (4).
Then
 $$\| P_{\alpha}\|_{\cal B}= \frac{\Gamma(n+N+\alpha+1)\Gamma(\frac{1+N}{2})}{\Gamma(\frac{1+N}{2}+\alpha+n)}+\frac{\Gamma(n+N+\alpha+1)\Gamma(N)}{\Gamma^{2}(\frac{N}{2}+\frac{n+\alpha+1}{2})},\enspace N\in{\rm{\bf N}}.$$

\end{theorem}
Let us notice that  when $N=1$  for the ${\tilde{\cal B}}$-norm of the Bergman projection
   we have $\| P_{\alpha}\|_{\tilde{\cal B}}=\frac{\Gamma(n+\alpha+2)}{\Gamma^{2}(\frac{n+\alpha+2}{2})}$ and this is one of the main results in \cite{Kalaj}. For the special case $n=1$, we obtain $\| P\|_{\tilde{\cal B}}=\frac{8}{\pi},$  which coincides with the main result of Per\"al\"a in \cite{Ant}.

\section{Proof of Theorem 4 and Theorem 5 }
Before we start to prove Theorem 4, let us state \cite[Lemma~3.3]{Kalaj}.
\begin{lemma}
For a multi-index $m=(m_1,...,m_n)\in {\rm{\bf N}_{0}^{n}}$ we have
\begin{equation}
\int_{S}|\zeta^{m}|d\sigma(\zeta)=\frac{(n-1)!\prod_{i=1}^{n}\Gamma(1+\frac{m_i}{2})}{\Gamma(n+\frac{|m|}{2})}
\end{equation}
and
\begin{equation}
\int_{\Bbb B}|z^{m}|dv_{\alpha}(z)=\frac{\Gamma(1+\alpha+n)}{\Gamma(1+\alpha+n+\frac{|m|}{2})}\prod_{i=1}^{n}\Gamma(1+\frac{m_i}{2})
\end{equation}
Here $w^{m}:=\prod_{i=1}^{n}w_{i}^{m_i},\enspace\mbox{and}\enspace |m|=\sum_{i=1}^{n}m_i.$
\end{lemma}

\begin{proof}

Let $P$ be the Bergman projection, $P:L^{\infty}(\Bbb B)\rightarrow {\cal B}.$ Since $P$ is onto, for any $f\in {\cal B}$ there is $g\in L^{\infty}(\Bbb B)$ such that $f=Pg,$ i.e.

\begin{equation}
f(z)=\int_{\Bbb B}\frac{g(w)}{(1-\left<z,w\right>)^{n+1+\alpha}}dv_{\alpha}(w),z\in \Bbb B.
\end{equation}
Differenting under integral sign in (10) we obtain
\begin{equation}\label{pon1}
\begin{split}
&\|P_{\alpha}g\|_{\tilde{\cal B}}=\max_{|m|=N}\sup_{z\in \Bbb B}{(1-|z|^2)^N}\left|\frac{\partial^N f(z)}{\partial z^m}\right|\\
\leq\| g\|_{\infty}&\frac{\Gamma(n+N+\alpha+1)}{\Gamma(n+\alpha+1)}\max_{|m|=N} \sup_{z\in B_n}{(1-|z|^2)^N}
  \int_{B_n}\frac{|h_{m}(w)|}{|1-\left<z,w\right>|^{n+1+N+\alpha}}dv_{\alpha}(w).
\end{split}
\end{equation}
Thus, we have

\begin{equation}\label{pon2}
\|P_{\alpha}\|_{\tilde{\cal B}}\leq  \frac{\Gamma(n+N+\alpha+1)}{\Gamma(n+\alpha+1)}\max_{|m|=N} \sup_{z\in\Bbb B}{(1-|z|^2)^N}
  \int_{B_n}\frac{|h_{m}(w)|}{|1-\left<z,w\right>|^{n+1+N+\alpha}}dv_{\alpha}(w),
\end{equation}

  where $h_m(\bar{w})={\bar {w}}^m=({\bar w}_1)^{m_1}...({\bar w}_n)^{m_n},\sum m_n=N.$

For a fixed $z\in \Bbb B$ let us make the change of variable $w=\varphi_{z}(\omega).$
By using the following formula for the real Jacobian $$(J_{R}\varphi_z)(\omega)=\left(\frac{1-|z|^2}{|1-\left<z,\omega\right>|^{2}}\right)^{n+1},$$
 and identity (1) we obtain
\begin{equation}\label{pon3}
\begin{split}
 &dv_{\alpha}(w)=c_{\alpha}(1-|w|^2)^{\alpha}dv(w)=c_{\alpha}\left(\frac{1-|z|^2}{|1-\left<z,\omega\right>|^2}\right)^{n+1}\frac{(1-|z|^2)^{\alpha}(1-|\omega|^2)^{\alpha}}{|1-\left<z,\omega\right>|^{2\alpha}}dv(\omega)
\end{split}
\end{equation}
By plugging \eqref{pon3} in \eqref{pon2}  we obtain
\begin{equation}
\begin{split}
\|P_{\alpha}\|_{\tilde{\cal B}}&\leq \frac{\Gamma(n+N+\alpha+1)}{\Gamma(n+\alpha+1)}\max_{|m|=N}\sup_{z\in \Bbb B}{(1-|z|^2)^N}\int_{\Bbb B}\frac{|h_{m}(w)|}{|1-\left<z,w\right>|^{n+N+\alpha+1}}dv_{\alpha}(w)\\
&=\frac{\Gamma(n+N+\alpha+1)}{\Gamma(n+\alpha+1)}\max_{|m|=N}\sup_{z\in \Bbb B}\int_{\Bbb B}\frac{|h_{m}(\varphi_{z}(\omega))|}{|1-\left<z,\omega\right>|^{n-N+\alpha+1}}dv_{\alpha}(\omega).
\end{split}
\end{equation}
Furthermore
\begin{equation}
\begin{split}
&\|P_{\alpha}\|_{\tilde{\cal B}}\leq \frac{\Gamma(n+N+\alpha+1)}{\Gamma(n+\alpha+1)}\max_{|m|=N}\sup_{z\in \Bbb B}\int_{\Bbb B}\frac{|h_{m}(\varphi_{z}(\omega))|}{|1-\left<z,\omega\right>|^{n-N+\alpha+1}}dv_{\alpha}(\omega)\\
&\leq \frac{\Gamma(n+N+\alpha+1)}{\Gamma(n+\alpha+1)}\max_{|m|=N}\left(\sup_{\omega\in \Bbb B}|h_{m}(\omega)|\right)\sup_{z\in \Bbb B}\int_{\Bbb B}\frac{1}{|1-\left<z,\omega\right>|^{n-N+\alpha+1}}dv_{\alpha}(\omega).
\end{split}
\end{equation}
Further, let us note that for every polynomial $h_{m},$ $|h_{m}(\omega)|\leq 1.$  The maximal value is attained, for example when $h_{(N,0,...,0)}(\omega)=h_{1}(\omega)=\omega_{1}^{N}$ and $\omega=e_1.$
So we conclude
\begin{equation}
\|P_{\alpha}\|_{\tilde{\cal B}}\leq \frac{\Gamma(n+N+\alpha+1)}{\Gamma(n+\alpha+1)}\sup_{z\in \Bbb B}\int_{\Bbb B}\frac{1}{|1-\left<z,\omega\right>|^{n-N+\alpha+1}}dv_{\alpha}(\omega).
\end{equation}

Our next goal is to determine maximum of the function $m(z),$ where
 $$m(z)=\frac{\Gamma(n+N+\alpha+1)}{\Gamma(n+\alpha+1)}\int_{\Bbb B}\frac{dv_{\alpha}(\omega)}{|1-\left<z,\omega\right>|^{n-N+\alpha+1}},\enspace z\in \Bbb B.$$

By using the uniform convergence, the fact that $\left<z,\omega\right>^{k_1}$ and $\left<z,\omega\right>^{k_2}\enspace(k_1,k_2\in {\rm{\bf N}},k_1\neq k_2)$
 are orthogonal in $L^{2}(\Bbb B, dv_{\alpha}(\omega)),$ and polar coordinates we obtain the following sequence of equalities
\begin{equation}
\begin{split}
& m(z)= \frac{\Gamma(n+N+\alpha+1)}{\Gamma(n+\alpha+1)}\int_{\Bbb B}\frac{dv_{\alpha}(\omega)}{|1-\left<\zeta,\omega\right>|^{n-N+\alpha+1}}dv_{\alpha}(\omega)\\
&=\frac{\Gamma(n+N+\alpha+1)}{\Gamma(n+\alpha+1)} \sum_{k=0}^{\infty}\left|\frac{\Gamma(k+\lambda)}{k!\Gamma(\lambda)}\right|^{2}\int_{\Bbb B}|\left<z,\omega\right>|^{2k}dv_{\alpha}(\omega)\\
&=\frac{2n\Gamma(n+N+\alpha+1)}{n!\Gamma(\alpha+1)} \sum_{k=0}^{\infty}\left|\frac{\Gamma(k+\lambda)}{k!\Gamma(\lambda)}\right|^{2}\int_{0}^{1}r^{2n+2k-1}(1-r^2)^{\alpha}dr\int_{S}|\left<z,\xi\right>|^{2k}d\sigma(\xi)
\end{split}
\end{equation}

where $\lambda=\frac{n-N+\alpha+1}{2}$ and $\omega=r\xi,|\xi|=1.$

 The unitary matrix  $U,\enspace U\xi=\xi'(\xi'=(\xi_{1}',...,\xi_{n}'), \xi_{1}'=\frac{\left<\xi,z\right>}{|z|})$ constructed in Zhu (see~\cite[P.~15]{Zhu}) applied on the last surface integral gives
\begin{equation}
m(z)= \frac{\Gamma(n+N+\alpha+1)}{n!\Gamma(\alpha+1)}\sum_{k=0}^{\infty}\left|\frac{\Gamma(k+\lambda)}{k!\Gamma(\lambda)}\right|^{2}\frac{n\Gamma(n+k)\Gamma(\alpha+1)}{\Gamma(n+k+\alpha+1)}\int_{S} |\xi_{1}'|^{2k}d\sigma(\xi')|z|^{2k}.
\end{equation}
Finally, by Lemma 6 we have

\begin{equation}
\begin{split}
m(z)&=\frac{\Gamma(n+N+\alpha+1)}{\Gamma^2(\lambda)}\sum_{k=0}^{\infty}\frac{\Gamma^2(k+\lambda)}{k!\Gamma(n+k+\alpha+1)}|z|^{2k}\\
&=\frac{\Gamma(n+N+\alpha+1)}{\Gamma(n+\alpha+1)}{}_2F_{1}(\lambda; \lambda;n+\alpha+1,|z|^2),
\end{split}
\end{equation}
where ${}_2F_{1}(\lambda; \lambda;n+\alpha+1,|z|^2)$ is the hypergeometric function, i.e. in general
$${}_2F_{1}(a;b;c,x)=\sum_{n=0}^{\infty}\frac{(a)_{n}(b)_{n}}{n!(c)_{n}}x^n$$
where $(a)_{n}=a(a+1)\cdot\cdot\cdot(a+n-1)$ is the Pochhammer symbol (see~\cite{Georg}). By using the formula
$$\frac{d}{dx}{}_2F_{1}(a,b;c;x)=\frac{ab}{c}{}_2F_{1}(a+1,b+1;c+1;x),$$
we conclude that the maximum of ${}_2F_{1}(\lambda; \lambda;n+\alpha+1,|z|^2)$ is ${}_2F_{1}(\lambda; \lambda;n+\alpha+1,1).$ So
\begin{equation}
\begin{split}
\sup_{z\in \Bbb B}m(z)&=\frac{\Gamma(n+N+\alpha+1)}{\Gamma(n+\alpha+1)}{}_2F_{1}(\lambda; \lambda;n+\alpha+1,1)\\
&=\frac{\Gamma(n+N+\alpha+1)}{\Gamma(n+\alpha+1)}\frac{\Gamma(n+\alpha+1)\Gamma(N)}{\Gamma^2(\frac{N}{2}+\frac{n+\alpha+1}{2})},
\end{split}
\end{equation}
i.e.
\begin{equation}
\|P_{\alpha}\|_{\tilde{\cal B}}\le\frac{\Gamma(n+N+\alpha+1)\Gamma(N)}{\Gamma^{2}(\frac{N}{2}+\frac{n+\alpha+1}{2})},\enspace N\in{\rm{\bf N}}
\end{equation}
In the relation (20) we used the Gauss identity for hypergeometric functions. Namely, for ${\rm Re}(c-a-b)>0,$ we have
$$ {}_2F_{1}(a;b;c,1)=\frac{\Gamma(c)\Gamma(c-a-b)}{\Gamma(c-a)\Gamma(c-b)}.$$
Let us prove the opposite inequality. Since the function $|h_{m}(\omega)|$ is subharmonic in $\Bbb B,$ there exists $\zeta_{0}\in S$  such that
$$\max_{|\zeta|=1}|h_{m}(\zeta)|=|h_{m}(\zeta_{0})|.$$
As we already pointed out if $h_{k}(\omega)=\omega_{k}^{N}$ and $\zeta_{0}=e_{k}$, $(h_{k}(\omega)=h_{(0,..,N,..,0)}(w))$, then $|h_{k}(\zeta_{0})|=1.$
We fix $z_{r}=r\zeta_{0},$ and the function $g_{z_r}(w)=\frac{(1-\left<z_r,w\right>)^{n+N+\alpha+1}}{|1-\left<z_r,w\right>|^{n+N+\alpha+1}}.$ It is clear that $\| g_{z_r}\|_{\infty}=1.$
Then
\begin{equation}
\begin{split}
\| P_{\alpha}g_{z_{r}}\|_{\tilde{\cal B}}&=\frac{\Gamma(N+n+\alpha+1)}{\Gamma(n+\alpha+1)}\max_{|m|=N}\sup_{z\in\Bbb B}(1-|z|^2)^{N}\left|\int_{\Bbb B}\frac{g_{z_{r}}(w)h_{m}(w)dv_{\alpha}(w)}{(1-\left<z,w\right>)^{n+N+\alpha+1}}\right|\\
&\geq\frac{\Gamma(N+n+\alpha+1)}{\Gamma(n+\alpha+1)}\max_{|m|=N}(1-|z_{r}|^2)^{N}\left|\int_{\Bbb B}\frac{h_{m}(w)dv_{\alpha}(w)}{|1-\left<z_{r},w\right>|^{n+N+\alpha+1}}\right|.
\end{split}
\end{equation}
By using the change of variable, $w\rightarrow \varphi_{z_{r}}(\omega),$ as in the previous case we have
\begin{equation}
\| P_{\alpha}g_{z_{r}}\|_{\tilde{\cal B}}\geq
\frac{\Gamma(N+n+\alpha+1)}{\Gamma(n+\alpha+1)}\max_{|m|=N}\left|\int_{\Bbb B}\frac{h_{m}(\varphi_{z_{r}}(w))dv_{\alpha}(w)}{|1-\left<z_{r},w\right>|^{n-N+\alpha+1}}\right|.
\end{equation}
Since
 $$\left|\int_{\Bbb B}\frac{h_{m}(\varphi_{z_{r}}(\omega))dv_{\alpha}(\omega)}{|1-\left<z_{r},w\right>|^{n-N+\alpha+1}}\right|\leq \int_{\Bbb B}\frac{dv_{\alpha}(\omega)}{|1-\left<z_{r},\omega\right>|^{n-N+\alpha+1}}<\infty,$$ we can apply the Lebesgue dominated convergence theorem in order to obtain
 \begin{equation}
 \begin{split}
 \| P_{\alpha}\|_{\tilde{\cal B}}&\geq\lim_{r\rightarrow 1^{-}}\frac{\Gamma(N+n+\alpha+1)}{\Gamma(n+\alpha+1)}\max_{|m|=N}\left|\int_{\Bbb B}\frac{h_{m}(\varphi_{z_{r}}(w))dv_{\alpha}(w)}{|1-\left<z_{r},w\right>|^{n-N+\alpha+1}}\right|\\
 &=\frac{\Gamma(N+n+\alpha+1)}{\Gamma(n+\alpha+1)}\max_{|m|=N}\left|\int_{\Bbb B}\frac{h_{m}(\zeta_{0})dv_{\alpha}(w)}{|1-\left<\zeta_{r},w\right>|^{n-N+\alpha+1}}\right|.
 \end{split}
 \end{equation}
We used in (24) that $\varphi_{\zeta_{0}}(w)=\zeta_{0}$ when $|\zeta_{0}|=1.$ Finally, from (24) we obtain

\begin{equation}
\begin{split}
\|P_{\alpha}\|_{\tilde{\cal B}}&\geq\frac{\Gamma(N+n+\alpha+1)}{\Gamma(n+\alpha+1)}|h_{m}(\zeta_{0})|\left|\int_{\Bbb B}\frac{dv_{\alpha}(w)}{|1-\left<\xi,w\right>|^{n-N+\alpha+1}}\right|\\
&=\frac{\Gamma(n+N+\alpha+1)\Gamma(N)}{\Gamma^{2}(\frac{N}{2}+\frac{n+\alpha+1}{2})}.
\end{split}
\end{equation}

\end{proof}
Now we prove Theorem 5.
\begin{proof}
We use the same notation as in the proof of Theorem 4. Let $f(z)=P_{\alpha}(g)(z), z\in \Bbb B,$ where $g\in L^{\infty}(\Bbb B),f\in{\cal B}.$\\
Then
\begin{equation}
\begin{split}
\|P_{\alpha}g\|_{\cal B}&= \max_{|m|\leq N-1}{\left|\frac{\partial^{|m|}f}{\partial z^{m}}(0)\right|} +\max_{|m|=N}{\sup_{z\in {\Bbb B}}(1-|z|^2)^{N}\left|\frac{\partial^{N}f}{\partial z^{m}}(z)\right|}\\
&\leq \|g\|_{\infty}\max_{|m|\leq N-1}\int_{\Bbb B}|h_{m}(w)|dv_{\alpha}(w)+\|g\|_{\infty}\|P\|_{\tilde{\cal B}},
\end{split}
\end{equation}
i.e.,
$$\|P\|_{\cal B}\leq \max_{|m|\leq N-1}\int_{\Bbb B}|h_{m}(w)|dv_{\alpha}(w)+\|P\|_{\tilde{\cal B}}.$$
By using Lemma 6 and the polar coordinates, we obtain
\begin{equation}
\begin{split}
\|P\|_{\cal B}&\leq \max_{|m|\leq N-1}\frac{\Gamma(|m|+n+\alpha+1)}{\Gamma(\frac{|m|}{2}+\alpha+n+1)}\prod_{j=1}^{n}\Gamma(1+\frac{m_{j}}{2})+\|P\|_{\tilde{\cal B}}\\
&=\frac{\Gamma(n+N+\alpha+1)\Gamma(\frac{1+N}{2})}{\Gamma(\frac{1+N}{2}+\alpha+n)}+\frac{\Gamma(n+N+\alpha+1)\Gamma(N)}{\Gamma^{2}(\frac{N}{2}+\frac{n+\alpha+1}{2})}.
\end{split}
\end{equation}
In order to prove the opposite inequality we make use of the functions  $$g_{z_{r}}(w)=\frac{(1-\left<z_r,w\right>)^{n+N+\alpha+1}}{|1-\left<z_r,w\right>|^{n+N+\alpha+1}},\enspace w\in \Bbb B$$
 which we used in the proof of Theorem~4 to maximize $\|P_{\alpha}f\|_{\tilde{\cal B}}.$ We define new test functions $g_{z_r}^{\delta}$ with $\|g_{z_r}^{\delta}\|_{\infty}\leq 1$ as follows:

  $$g_{z_r}^{\epsilon}(w)=\left\{\begin{array}{rl}
                                     g_{z_r}(w), &|w|\geq \delta \\
              \frac{w_{1}^{N-1}}{|w_{1}|^{N-1}}, &|w|\leq {\delta}^{2}
                     \end{array}
                     \right.$$\\

 and define $g_{z_r}^{\delta}$ on $\{{\delta}^2<|w|<\delta\}$ so that $g_{z_r}^{\delta}$ is continuous on ${\Bbb B}.$   \\                 We claim that $$(1-|z_r|^2)^{N}\max_{|m|=N}{\left|\frac{\partial^{N}Pg_{z_r}^{\delta}}{z^{m}}(z_r)\right|}\rightarrow \|P\|_{\tilde{\cal B}},\enspace\mbox{as}\enspace r\rightarrow 1^{-}.$$
 Namely, it is clear by the definition of the semi-norm $\|\cdot\|_{\cal B}$ that
$$\limsup_{r\rightarrow 1^{-}}(1-|z_r|^2)^{N}\max_{|m|=N}{\left|\frac{\partial^{N}Pg_{z_r}^{\delta}}{z^{m}}(z_r)\right|}\leq \|P\|_{\tilde{\cal B}}.$$
Also, we have shown in the proof of  Theorem~4 that
$$\lim_{r\rightarrow 1^{-}}(1-|z_r|^2)^{N}\left|\frac{\partial^{N}P_{\alpha}g_{z_r}}{z^m}(z_r)\right|=\|P_{\alpha}\|_{\tilde{\cal B}}.$$
Since $|g_{z_r}(w)-g_{z_r}^{\delta}(w)|\leq 2$ on $\Bbb B$ and $|g_{z_r}(w)-g_{z_r}^{\delta}(w)|=0$ when $|w|>\delta,$ we have
\begin{equation}
\begin{split}
(1-|z_r|^2)^{N}&\max_{|m|=N}\left|\frac{\partial^{N}P_{\alpha}g_{z_r}}{z^{m}}(z_r)-\frac{\partial^{N}P_{\alpha}g_{z_r}^{\delta}}{z^{m}}(z_r)\right|\\
 &=(1-|z_r|^2)^{N}\max_{|m|=N}{\left|\frac{\partial^{N}P_{\alpha}(g_{z_r}-g_{z_r}^{\delta})}{z^{m}}(z_r)\right|}\\
&\leq\frac{\Gamma(n+N+\alpha+1)}{\Gamma(n+\alpha+1)}\int_{|w|<\delta}\frac{2(1-|z_r|^2)^{N}dv_{\alpha}(w)}{|1-\left<z_{r},w\right>|^{n+N+\alpha+1}}.
\end{split}
\end{equation}
The right hand side in (28) goes to 0 as $r\rightarrow 1^{-}.$\\
 Thus
\begin{equation}
\begin{split}
\lim_{r\rightarrow 1^{-}}(1-|z_r|^2)^{N}&\max_{|m|=N}{\left|\frac{\partial^{N}P_{\alpha}g_{z_r}^{\delta}}{z^{m}}(z_r)\right|}\\
&=\lim_{r\rightarrow 1^{-}}(1-|z_r|^2)^{N}\max_{|m|=N}{\left|\frac{\partial^{N}P_{\alpha}g_{z_r}}{z^{m}}(z_r)\right|}\\
&=\|P_{\alpha}\|_{\tilde{\cal B}}.
\end{split}
\end{equation}
Furthermore, for every $r\in(0,1),$ we have
\begin{equation}
\begin{split}
&\left|\frac{\partial^{N-1}P(g_{z_r}^{\delta})}{z_{1}^{N-1}}(0)\right|
\geq\int_{|w|\leq \delta^{2}}|w_1|^{N-1}dv_{\alpha}(w)-\int_{|w|>\delta^2}dv_{\alpha}\\
&\rightarrow\int_{\Bbb B}|w_1|^{N-1}dv_{\alpha}(w)=\frac{\Gamma(n+N+\alpha+1)\Gamma(\frac{1+N}{2})}{\Gamma(\frac{1+N}{2}+\alpha+n)}
\end{split}
\end{equation}
as $\delta\rightarrow 1^{-}.$ It is clear that in (30) we might observe any partial derivative $\frac{\partial^{N-1}P(g_{z_r}^{\delta})}{z_{k}^{N-1}}(0),$ where $k=1,...,n.$ \\
For given $\epsilon>0,$ we may pick $\delta >0$ such that
$$\left|\frac{\partial^{N-1}Pg_{z_r}^{\delta}}{z_{1}^{N-1}}(0)\right|>\frac{\Gamma(n+N+\alpha+1)\Gamma(\frac{1+N}{2})}{\Gamma(\frac{1+N}{2}+\alpha+n)}-\frac{\epsilon}{2},$$
 for every $r\in (0,1).$ We fix such $\delta.$ According to the relation (29), one can pick $r\in (0,1)$ such that
 $$(1-|z_r|^2)^{N}\max_{|m|=N}{\left|\frac{\partial^{N}P_{\alpha}g_{z_r}^{\delta}}{z^{m}}(z_r)\right|}>\|P_{\alpha}\|_{\tilde{\cal B}}-\frac{\epsilon}{2}.$$

 Then, we can end up with a function $g_{z_{r}}^{\delta}$ such that
\begin{equation}
\begin{split}
\|P_{\alpha}\|_{\cal B}&\geq\|P_{\alpha}g_{z_r}^{\delta}\|_{\cal B}\\
&\geq\left|\frac{\partial^{N-1}P(g_{z_r}^{\delta})}{z_{1}^{N-1}}(0)\right|+(1-|z_r|^2)^{N}\max_{|m|=N}{\left|\frac{\partial^{N}P_{\alpha}g_{z_r}^{\delta}}{z^{m}}(z_r)\right|}\\
&>\frac{\Gamma(n+N+\alpha+1)\Gamma(\frac{1+N}{2})}{\Gamma(\frac{1+N}{2}+\alpha+n)}+\|P_{\alpha}\|_{\tilde{\cal B}}-\epsilon.
\end{split}
\end{equation}
Therefore, $\|P_{\alpha}\|_{\cal B}\geq\frac{\Gamma(n+N+\alpha+1)\Gamma(\frac{1+N}{2})}{\Gamma(\frac{1+N}{2}+\alpha+n)}+\|P_{\alpha}\|_{\tilde{\cal B}},$ and combining with relation (27) proves the theorem.
\end{proof}

\begin{remark}
If  $P_{\alpha}$ is Bergman projection, $P_{\alpha}:L^{\infty}(\Bbb B)\rightarrow {\cal B},$ where ${\cal B}$ is Bloch space in the norm (4),
then it is easy to find the lower estimate for the $\cal B$-norm of $P_{\alpha}$ i.e.
$$\| P_{\alpha}\|_{\cal B}\geq \frac{\Gamma(N+n+\alpha+1)}{\Gamma(n+\alpha+1)}.$$

Namely, we fix $z_{0}\in\Bbb  B$ and we make use of the function $g_{z_0}(w)=\frac{(1-\left<z_{0},w\right>)^N}{(1-\left<w,z_{0}\right>)^N}.$ It is clear that
$g_{z_{0}}\in L^{\infty},\enspace \| g_{z_{0}}\|=1.$

Hence
\begin{equation}
\begin{split}
\| P_{\alpha}g_{z_{0}}\|_{\cal B}&=\frac{\Gamma(N+n+\alpha+1)}{\Gamma(n+\alpha+1)}\max_{|m|=N}\sup_{z\in \Bbb B}(1-|z|^2)^{N}\left|\int_{\Bbb B}\frac{g_{z_{0}}(w)h_{m}(w)}{(1-\left<z,w\right>)^{n+N+\alpha+1}}dv_{\alpha}(w)\right|\\
&\geq\frac{\Gamma(N+n+\alpha+1)}{\Gamma(n+\alpha+1)} \max_{|m|=N} (1-|z_{0}|^2)^{N}\left|\int_{\Bbb B}\frac{\frac{h_{m}(w)}{(1-\left<w,z_{0}\right>)^N}}{(1-\left<z_{0},w\right>)^{n+\alpha+1}}dv_{\alpha}(w)\right|.
\end{split}
\end{equation}
 On the other hand, it holds $\frac{h_{m}(w)}{(1-\left<w,z_{0}\right>)^N}\in H^{\infty}(\Bbb B),$ and this implies

\begin{equation}
\begin{split}
\| P_{\alpha}\|_{\cal B}&\geq \frac{\Gamma(N+n+\alpha+1)}{\Gamma(n+\alpha+1)}\max_{|m|=N}\sup_{z\in \Bbb B} |h_{m}(z)|\\
&=\frac{\Gamma(N+n+\alpha+1)}{\Gamma(n+\alpha+1)}.
\end{split}
\end{equation}

\end{remark}

\begin{remark}
We want to emphasize that on the Bloch space $\cal B$ we may observe the norm
\begin{equation}
\|f\|_{\cal B_{p}}=\sum_{|m|\leq N-1}\left|\frac{\partial^{|m|}f}{\partial z^{m}}(0)\right|+\sup_{z\in \Bbb B}(1-|z|^2)^{N}\left(\sum_{|m |=N}\left|\frac{\partial^{N}f}{\partial z^{m}}(z)\right|^{p}\right)^{\frac{1}{p}},
\end{equation} and
the semi-norm
\begin{equation}
\|f\|_{\beta_{p}}=\sup_{z\in \Bbb B}(1-|z|^2)^{N}\left(\sum_{|m |=N}\left|\frac{\partial^{N}f}{\partial z^{m}}(z)\right|^{p}\right)^{\frac{1}{p}}
\end{equation}
where $f\in{\cal B},N \in {\rm{\bf N}},\enspace 1\leq p<\infty.$\\

By the same argument as in the proof of  Theorem 4 it can be shown that
$$\|P_{\alpha}\|_{\beta_{p}}\leq\frac{\Gamma(n+N+\alpha+1)\Gamma(N)}{\Gamma^{2}(\frac{N}{2}+\frac{n+\alpha+1}{2})}\left(\sum_{|m|=N}\max_{|\zeta|=1}|h_{m}(\zeta)|^{p}\right)^{\frac{1}{p}}.$$
In particular, when $p=2$ and $N=1$ we have
\begin{equation}
\begin{split}
\|f\|_{\beta_{2}}&=\sup_{z\in \Bbb B}(1-|z|^2)^{N}\left(\sum_{m=1}^{n}\left|\frac{\partial f}{\partial z_{1}}(z)\right|^{2}\right)^{\frac{1}{2}}\\
&=\sup_{z\in \Bbb B}(1-|z|^2)^{N}|\nabla f(z)|
\end{split}
\end{equation}

and
\begin{equation}
\begin{split}
&\|P_{\alpha}\|_{\beta_{2}}\leq\frac{\Gamma(n+N+\alpha+1)\Gamma(N)}{\Gamma^{2}(\frac{N}{2}+\frac{n+\alpha+1}{2})}\left(\sum_{|m|=N}\max_{|\zeta|=1}|h_{m}(\zeta)|^{2}\right)^{\frac{1}{2}}\\
&\leq \frac{\Gamma(n+N+\alpha+1)\Gamma(N)}{\Gamma^{2}(\frac{N}{2}+\frac{n+\alpha+1}{2})}\sqrt{n}.
\end{split}
\end{equation}
\end{remark}

\end{document}